\documentclass[11pt]{amsart}
\usepackage{a4wide}
\usepackage{amsmath}
\usepackage{amsfonts}
\usepackage{amssymb, mathrsfs}
\usepackage{enumerate}
\usepackage{color}
\usepackage{graphicx}
\usepackage{comment}
\newcommand{\beq}{\begin{equation}}
\newcommand{\eeq}{\end{equation}}

\def\R{\mathbb R}
\def\N{\mathbb N}
\def\C{\mathcal{C}}

\def\cal{\mathcal}

\def\H{{\cal H}}

\def\e{\varepsilon}

\def\om{\omega}
\def\Chi#1{\hbox{{\large $\chi$}{\Large $_{_{#1}}$}}}
\newcommand{\medint}{-\kern -,375cm\int}
\newcommand{\medintinrigo}{-\kern -,315cm\int}

\def\pa{\partial}

\newcommand{\eps}{\varepsilon}

\providecommand{\U}[1]{\protect\rule{.1in}{.1in}}
\newtheorem{theorem}{Theorem}

\newtheorem{corollary}[theorem]{Corollary}

\newtheorem{lemma}[theorem]{Lemma}

\newtheorem{proposition}[theorem]{Proposition}
\theoremstyle{definition}
\begingroup

\newtheorem{remark}[theorem]{Remark}

\endgroup
\numberwithin{equation}{section}
\numberwithin{theorem}{section}

\title[Exterior isoperimetric on convex sets]{Rigidity and large volume residues \\ in exterior isoperimetry for convex sets}
\author{N. Fusco} 
\address{Dipartimento di Matematica e Applicazioni ``R. Caccioppoli",
Universit\`{a} degli Studi di Napoli ``Federico II" , Napoli, Italy}
\email{n.fusco@unina.it}

\author{F. Maggi}
\address{Department of Mathematics, The University of Texas at Austin, 2515 Speedway, Stop C1200, Austin TX 78712-1202, United States of America}
\email{maggi@math.utexas.edu}

\author{M. Morini}
\address{Dipartimento di Scienze Matematiche Fisiche e Informatiche, Universit\`{a} degli Studi di Parma, Parma, Italy}
\email{massimiliano.morini@unipr.it}

\author{M. Novack}
\address{Department of Mathematical Sciences, Carnegie Mellon University, Wean Hall 6113,
Pittsburgh, PA 15213, United States of America}
\email{mnovack@andrew.cmu.edu}

\begin{document}
\allowdisplaybreaks


\begin{abstract} A comparison theorem by Choe, Ghomi and Ritor\'e states that the exterior isoperimetric profile $I_\C$ of any convex body $\C$ in $\R^N$ lies above that of any half-space $H$. We characterize convex bodies such that $I_\C\equiv I_H$ in terms of a notion of ``maximal affine dimension at infinity'', briefly called the asymptotic dimension $d^*(\C)$ of $\C$. More precisely, we show that $I_\C\equiv I_H$ if and only if $d^*(\C)\ge N-1$. We also show that if $d^*(\C)\le N-2$, then, for large volumes, $I_\C$ is asymptotic to the isoperimetric profile of $\R^N$. We then estimate, in terms of $d^*(\C)$-dependent power laws, the order as $v\to\infty$ of the difference between $I_\C$ and the isoperimetric profile of $\R^N$.
\end{abstract}

\maketitle

\tableofcontents

\section{Introduction} In this paper we consider the exterior isoperimetric problem for convex bodies $\C\subset \R^N$. We first identify a natural notion of ``maximal affine dimension at infinity'', called the {\it asymptotic dimension} of $\C$, denoted by $d^*(\C)$ and typically larger than the affine dimension of the recession cone $\C_\infty$ of $\C$. We then prove that unbounded convex bodies with $d^*(\C)\ge N-1$ have the {\it same} exterior isoperimetric profile of half-spaces. We also begin the study of exterior isoperimetric profiles of convex bodies with $d^*(\C)\le N-2$, and introduce in this setting the notion of {\it isoperimetric residue}. We have two motivations for presenting these results. The first one is the desire of understanding the geometric information ``stored'' in the large volume behavior of exterior isoperimetric profiles, as done in \cite{MN} with the introduction of the isoperimetric residues of compact (non-necessarily convex) sets. The second one is providing characterizations of rigidity of equality cases in a recent comparison theorem for exterior isoperimetric profiles of convex bodies proved by  Choe, Ghomi and Ritor\'e in \cite{CGR}. 

\medskip

Before further discussing these points we need to introduce some notation and terminology. Given a closed set $\C\subset\R^N$, the {\bf exterior isoperimetric profile} of $\C$ is the function $I_\C:(0,\infty)\to(0,\infty)$ defined by
\begin{equation}
\label{ICV}
I_\C(v)=\inf\big\{P(E;\R^N\setminus\C):E\subset\R^N\setminus\C\,,|E|=v\big\}\,,\qquad v>0\,.
\end{equation}
That is, $I_\C$ is what is commonly called the isoperimetric profile of $\R^N\setminus\C$; in particular, $I_\varnothing(v)=N\,\omega_N^{1/N}\,v^{(N-1)/N}$ coincides with the isoperimetric profile of $\R^N$. Here $|E|$ denotes the volume (Lebesgue measure) of a Borel set $E$ and $P(E;G)$ its distributional perimeter, so that $P(E;G)=\H^{N-1}(G\cap\partial E)$ as soon as $E$ is open with $C^1$-boundary. We also denote by $\H^k$ the $k$-dimensional Hausdorff measure in $\R^N$ and by $\omega_N$ the volume of the unit ball in $\R^N$.

\medskip

With this notation in hand, we can state the comparison theorem proved by Choe, Ghomi and Ritor\'e in \cite{CGR}: if $\C$ is a {\bf convex body}, that is, if $\C$ is a closed convex subset of $\R^N$ with nonempty interior and different from the whole $\R^N$, then, denoting by $H$ a generic half-space of $\R^N$, one has
\begin{equation}
\label{cgr comparison}
I_C(v)\ge I_H(v)=N\,\big(\omega_N/2\big)^{1/N}\,v^{(N-1)/N}\,,\qquad\forall v>0\,.
\end{equation}
To provide some context, we recall that {\it comparison theorems} of this form are one of the most studied type of results in Riemannian Geometry. For example, the Levy--Gromov comparison theorem states that among compact Riemannian manifolds $(M^N,g)$ whose Ricci tensor is bounded from below by $(N-1)\,g$, the standard sphere $(S^N,g_{S^N})$ has the lowest isoperimetric profile\footnote{The isoperimetric profile $\Phi_{(M,g)}(t)$ of $(M,g)$ for the volume fraction $t\in(0,1)$ is defined as the infimum of $P_g(E)/{\rm vol}_g(M)$ among all sets $E\subset M$ with $t={\rm vol}_g(E)/{\rm vol}_g(M)$.}. Comparison theorems are usually accompanied by {\it rigidity statements}. Using again the Levy--Gromov comparison theorem as a model, the corresponding rigidity statement is that if the isoperimetric profiles of $(M^N,g)$ and  $(S^N,g_{S^N})$ coincide even for just one volume fraction, then $(M^N,g)$ is isometric to   $(S^N,g_{S^N})$. The situation with the CGR-comparison theorem is remarkably different. Indeed, as proved in \cite{CGR} when $\C$ has boundary of class $C^2$, and for arbitrary convex bodies in \cite{FM}, one has
\begin{eqnarray}\label{rigidity 1}
&&\mbox{$I_\C(v)=I_H(v)$ for some value of $v>0$ s.t. there exist minimizers of $I_\C(v)$}
\\\nonumber
&&\hspace{4cm}\mbox{if and only if}
\\\label{rigidity 2}
&&\mbox{$\partial\C$ has a ``facet'' that supports a half-ball of volume $v$ contained in $\R^N\setminus\C$}\,.
\end{eqnarray}
In particular, the class of convex bodies satisfying $I_\C(v)=I_H(v)$ for a single value of $v>0$ is extremely vast, e.g., it contains all convex polyhedra!, and coincides with the class of convex bodies $\C$ such that $I_\C=I_H$ on an open interval. We can of course formulate the rigidity problem in a stronger sense, and consider the class of convex bodies $\C$ such that $I_\C(v)=I_H(v)$ for {\it every} $v>0$. While it is not hard to construct unbounded convex bodies that are not half-spaces and still satisfy $I_\C\equiv I_H$, there is a surprisingly direct condition that characterizes the stronger notion of rigidity $I_\C\equiv I_H$. This condition can be expressed in terms of the asymptotic dimension $d^*(\C)$ of $\C$ introduced right after the following statement, which is our first main result.

\begin{theorem}[Strong rigidity in the CGR-comparison theorem]
\label{theorem main 1}
Let $\C$ be a convex body in $\R^N$. Then $I_\C=I_H$ on $(0,\infty)$ if and only if $d^*(\C)\in\{N-1,N\}$. If, otherwise, $d^*(\C)\in\{0,1,...,N-2\}$, then 
\[
\lim_{v\to\infty}\frac{I_\C(v)}{N\,\omega_N^{1/N}\,v^{(N-1)/N}}=1\,,
\]
and, in particular, $I_\C(v)>I_H(v)$ for every $v$ sufficiently large (depending on $N$ and $\C$).
\end{theorem}

\begin{remark}
Notice that Theorem \ref{theorem main 1} addresses rigidity without making the conditional assumption that $I_\C(v)$ admits minimizers, either for one or more values of $v$. In fact, by combining Theorem \ref{theorem main 1} with the equivalence between \eqref{rigidity 1} and \eqref{rigidity 2} proved in \cite{CGR,FM}, one easily shows that {\it if $\C$ is strictly convex and $d^*(\C)\ge N-1$, then, for every $v>0$, $I_\C(v)$ does not admit minimizers}; see Corollary \ref{nonex}.
\end{remark}

We now introduce the notion of {\bf asymptotic dimension} $d^*(\C)$ of $\C$. To begin with we recall that if $x\in\C$ and $\C$ is a convex body, then the family of convex sets $\{\lambda\,(\C-x)\}_{\lambda>0}$ is monotone increasing with respect to set inclusion. In particular, its limit as $\lambda\to 0^+$, called the {\bf recession cone} $\C_\infty$ of $\C$, can be defined as
\begin{equation}
\label{recession cone}
\C_\infty=\bigcap_{\lambda>0}\lambda\,(\C-x)\,.
\end{equation}
The (affine) dimension $\dim(\C_\infty)$ of $\C_\infty$ (that is, the dimension of the smallest affine subspace of $\R^N$ containing $\C_\infty$) directly quantifies the dimension of the set of directions along which $\C$ is unbounded (and, indeed, $\C$ is bounded if and only if $\C_\infty$ is a point, that is, $\dim(\C_\infty)=0$). To define $d^*(\C)$, and understand rigidity of the CGR-comparison theorem, we follow a related procedure. Rather than working with a fixed point $x\in\C$, we consider sequences $\{x_n\}_n$ in $\C$ and $\{\lambda_n\}_n$ in $(0,\infty)$, look at the resulting sequences of convex sets $\{\lambda_n\,(\C-x_n)\}_n$, consider all their possible accumulation points $K$ in the Kuratowski convergence, which are still convex sets in $\R^N$, and finally {\it maximize} the affine dimension among the possible limits $K$. More formally and concisely, we set
\begin{equation}
\label{asymptotic dimension}
d^*(\C):=\max\big\{\dim(K):\mbox{$\exists\,x_n\in\C\,,\lambda_n\to 0$ s.t. $\lambda_n(\C-x_n)\to K$ as $n\to\infty$}\big\}
\end{equation}
where $\lambda_n(\C-x_n)\to K$ as $n\to\infty$ is meant in the sense of Kuratowski. By taking $x_n=x\in\C$ for every $n$ it is easily seen that $d^*(\C)\ge\dim(\C_\infty)$, and examples where this inequality is strict are easily constructed.
\medskip

Having completely discussed the rigidity problem for the CGR-comparison theorem, we move to the problem of understanding which information on the convex body $\C$ is stored in the behavior of $I_\C(v)$ as $v\to\infty$, and what can be said, should they exist, about large volume exterior isoperimetric sets of $\C$, i.e., about minimizers $E_v$ of $I_\C(v)$ as $v\to\infty$. 

\medskip

Theorem \ref{theorem main 1} plays an interesting role in setting up these problems. Indeed, Theorem \ref{theorem main 1} states that for a convex body $\C$ with $d^*(\C)\in \{N-1,N\}$ there is no  geometric information (in addition to  $d^*(\C)\in \{N-1,N\}$) that can be found by studying $I_\C$ for $v$ large, as $I_\C$ must then be identically equal to $I_H$. At the same time, when $d^*(\C)\le N-2$, Theorem \ref{theorem main 1} invites us to investigate what information about $\C$ may be stored in the behavior of the quantity
$$
\mathcal{R}_\C(v):=N\,\om_N^{1/N}\,v^{(N-1)/N} - I_\C(v)\,,
$$
as $v\to\infty$.

\medskip

When $d^*(\C)=0$, i.e., when $\C$ is bounded, this problem has been thoroughly addressed in the recent paper \cite{MN} without even assuming the convexity of $\C$. Roughly speaking, the main result in \cite{MN} is that if $\C$ is any compact set in $\R^N$, then
\[
\lim_{v\to\infty}\mathcal{R}_\C(v)=\mathcal{R}(\C)
\]
where $\mathcal{R}(\C)$, called the {\bf isoperimetric residue of $\C$}, is a non-negative quantity such that $\mathcal{R}(\lambda\,\C)=\lambda^{N-1}\,\mathcal{R}(\C)$ for every $\lambda>0$ and
\[
\sup_{M}\H^{N-1}(M\cap\C)\le\mathcal{R}(\C)\le\sup_M\H^{N-1}(\mathbf{p}_M(\C))
\]
where $M$ ranges over all the affine hyperplanes in $\R^N$ and $\mathbf{p}_M:\R^N\to M$ denotes the orthogonal projection over $M$. In fact, $\mathcal{R}(\C)$ can be characterized as an optimization problem whose solutions are area minimizing boundaries contained in $\R^N\setminus\C$, trapped in between two parallel hyperplanes, and intersecting $\C$ orthogonally. Moreover, exterior isoperimetric sets $E_v$ for $\C$ with $v$ large can be fully ``resolved as'' (i.e., expressed as small diffeomorphic deformations of) the union of a large sphere of volume $v$ missing a spherical cap and of an optimizing boundary in $\mathcal{R}(\C)$. Obtaining this resolution result requires the introduction a new $\varepsilon$-regularity criterion operating at mesoscales and interpolating between the local Allard's regularity theorem \cite{Allard}, based on the analysis of blowups, and the ``at infinity'' regularity theorems of Allard--Almgren \cite{AllardAlmgren} and Simon \cite{Simon}, based on the analysis of blowdowns.

\medskip

Coming back to the case of convex bodies, we are left to consider the case when $1\le d^*(\C)\le N-2$ (and thus, with $N\ge 3$). In this case we do not expect $\mathcal{R}_\C(v)$ to have a finite limit as $v\to\infty$, but rather to depend on $v$ through a power law. This expectation is not completely confirmed, but it is definitely supported, by our second main result. 

\begin{theorem}\label{theorem order of residue}
Let $\C$ be a convex body in $\R^N$, $N\ge 3$, with $d^*(\C)\in\{1,...,N-2\}$. Then there exists a positive constant $C_0$ depending on $N$ and $\C$ such that
\begin{equation}
\label{intro res}
\frac{v^{d^*(\C)/2N}}{C_0}\le \mathcal{R}_\C(v) \le C_0\,v^{d^*(\C)/N}\,,\qquad\forall v>C_0\,.
\end{equation}
\end{theorem}

It seems plausible to conjecture that the lower bound in \eqref{intro res} should be capturing the correct order of magnitude of $\mathcal{R}_\C(v)$. In other words, it seems plausible that for every $\C$ as in Theorem \ref{theorem order of residue} the limit
\[
\lim_{v\to\infty}\frac{\mathcal{R}_\C(v)}{v^{d^*(\C)/2N}}
\]
should exist in $(0,\infty)$, be amenable to be characterized as an optimization problem, and thus lead to extending the definition of isoperimetric residue to this class of obstacles. Resolving these issues would require fine geometric information on minimizers obtained by the application of an $\e$-regularity criterion analogous to \cite{MN}. Given the considerable complexity of this problem even in the compact case, we leave this question for future investigations; see Remark \ref{enveloping remark} for more discussion.

\medskip

We now discuss the organization of the paper and, in the process, we highlight some additional noteworthy statements proved in the paper but not included in Theorem \ref{theorem main 1} and Theorem \ref{theorem order of residue}. After setting our notation and recalling some basic facts concerning isoperimetric problems in Section \ref{section notation}, in Section \ref{section dstar} we collect several facts concerning the notion of asymptotic dimension: in particular, in Proposition \ref{general structure lemma} we describe the structure of sets with $d^*(\C)\le N-1$. In Section \ref{section large volume} we introduce a constrained exterior isoperimetric profile $I_{\C,R}$ (obtained by restricting competitors in $I_\C$ to be contained in $B_R(0)$), study its properties (Proposition \ref{prop:ICR}), and then use those for proving Theorem \ref{theorem main 1}. Finally, in Section \ref{section residues} we discuss the problem of defining isoperimetric residues for unbounded convex sets with asymptotic dimension less than $N-2$, and prove in particular Theorem \ref{theorem order of residue}, together with some properties of large volume exterior isoperimetric sets (see Theorem \ref{existence and uniform minimality} and Proposition \ref{cylinder hausdorff lemma take 2}) that lay the groundwork for further analysis.

\medskip

\noindent {\bf Acknowledgments:} FM has been supported by NSF Grant DMS-2247544. FM and MN have been supported by NSF Grant DMS-2000034 and NSF FRG Grant DMS-1854344. MN has been supported by NSF RTG Grant DMS-1840314. NF has been supported by PRIN Project 2017TEXA3H. The research of MM was partially supported by GNAMPA and by the University of Parma via the project “Regularity, Nonlinear Potential Theory and related topics”.

\section{Notation and preliminary results about perimeter minimizers}\label{section notation}
In the following we shall denote by $B_R(x)$ the ball with center at $x$ and radius $R$. If the center is at the origin we shall simply write $B_R$. Given $v>0$ 
we let $B^{(v)}(x)$ denote the ball of volume $v$ centered at $x\in \mathbb{R}^N$ and $B^{(v)}=B^{(v)}(0)$.

\medskip

Given $E\subset \R^N$ of locally finite perimeter and a Borel set $G$ we denote by $P(E;G)$ the perimeter of $E$ in $G$. The {\em reduced boundary} of $E$ will be denoted by $\pa^*E$, while $\pa^eE$ will stand for the {\em essential boundary} defined as
$$
\pa^eE:=\R^N\setminus(E^{(0)}\cup E^{(1)})\,,
$$
where $E^{(0)}$ and $E^{(1)}$ are the sets of points where the density of $E$ is $0$ and $1$, respectively. 
In the following, when dealing with a set of locally finite perimeter $E$, we shall always tacitly assume that $E$ coincides with a precise representative that satisfies the property $\pa E=\overline{\pa^*E}$, see \cite[Remark 16.11]{M}. A possible choice is given by $E^{(1)}$ for which one may easily check that
\beq\label{Euno}
\pa E^{(1)}=\overline{\pa^*E}\,.
\eeq

\par\medskip
We premise some lemmas. The first lemma is proved  for instance in \cite[Lemma 3.6]{FFLM} for  $N=3$, but the same statement (with the same argument) holds in any dimension.
\begin{lemma}\label{lm:density}
Let $\C\subset\R^N$ be a convex body  and let $E\subset B_R\setminus \C$ satisfy the following minimality property: there exists $\Lambda\geq0$ such that
\beq\label{density1}
P(E; \R^N\setminus \C)\leq P(F; \R^N\setminus \C)+\Lambda|F\Delta E| \qquad\text{for all $F\subset B_R\setminus \C$.}
\eeq
Then  $E$ is equivalent to an open set, still denoted by $E$, such that $\pa E=\pa^e E$ and hence $\H^{N-1}(\pa E\setminus\pa^* E)=0$. 
Moreover, there exist $c_0=c_0(N)>0$ and $r_0=r_0(N,\Lambda)\in(0,1)$ independent of $R$ and $\C$, such that if $x\in\pa E'$, $E'$ being a connected component of $E$,  then
\beq\label{density2}
 |E'\cap B_r(x)|\geq c_0r^N
\eeq
for every $0<r\leq r_0$. 
Finally, if \eqref{density1} holds for all $F\subset \R^N\setminus\C$, then the above density estimates holds also for $\R^N\setminus E'$; i.e.,  $|B_r(x)\setminus E'|\geq c_0r^N$ for every $0<r\leq r_0$. 
\end{lemma}

We recall that a sequence $\{F_n\}$ of closed sets converges in the {\em Kuratoswki sense} (or {\em locally in Hausdorff sense}) to a closed set $F$ if the following conditions are satisfied:
\begin{itemize}
\item[(i)] if $x_n\in F_n$ for every $n$, then any limit point of $\{x_n\}$ belongs to $ F$;
\item[(ii)] any $x\in F$ is the limit of a sequence $\{x_n\}$ with $x_n\in F_n$.
\end{itemize}
One can easily see that $F_n\to F$ in the sense of Kuratowski if and only if dist$(\cdot, F_n)\to$ dist$(\cdot, F)$ locally uniformly in $\R^N$. In particular, by the Arzel\`a-Ascoli Theorem any sequence of closed sets admits a subsequence which converge in the sense of Kuratowski.

\medskip

For the simple proof of the next lemma see for instance \cite[Remark~2.1]{FFLM}.
\begin{lemma}\label{Kuratowski equivalence}
Let $\{\C_n\}_n$ be a sequence of closed convex sets. Then $\C_n\to \C$ in the Kuratowski sense if and only if $\Chi{\C_n}\to \Chi{\C}$ pointwise almost everywhere. In addition, $\C$ is convex.
\end{lemma}

\noindent The next lemma is also well known, for the proof see for instance \cite[Lemma 5.1]{FM}.

\begin{lemma}\label{projection onto convex boundary}
Let $\C$ be a closed convex set with nonempty interior and $F\subset\mathbb{R}^N\setminus \C$ a bounded set of finite perimeter. Then
\begin{align}\notag
    P(F;\partial \C) \leq P(F;\mathbb{R}^N\setminus \C)\,.
\end{align}
\end{lemma}


The following result, see \cite[Lemma 2.1]{MPS}, is a simplified version of a nucleation lemma  due to Almgren \cite[VI.13]{A}; see also \cite[Lemma 29.10]{M}.

\begin{lemma}\label{lm:nonvanishing}
There exists a constant $c(N)>0$ such that if $E\subset\mathbb{R}^N$ is a set of finite perimeter and finite measure, then, setting $Q:=(0,1)^N$, we have 
$$
\sup_{z\in \mathbb{Z}^N}|E\cap (z+Q)|\geq c(N)\min\Bigl\{\Big(\frac{|E|}{P(E)}\Big)^N, 1\Bigr\}\,.
$$ 
\end{lemma}

\noindent Finally we conclude with a useful construction that allows one to locally dilate sets of finite perimeter with a controlled change in volume and perimeter. 
\begin{lemma}\label{lm:ef}
There exist $\eps(N), c(N)>0$ with the following property. Let $E\subset\R^N$ be a set of finite perimeter with $|E|<m$ and $B_R$ a ball such that
$$
|E\cap B_{\frac{R}{2}}|\leq\eps R^N,\quad |E\cap B_R|\geq \frac{\omega_NR^N}{2^{N+2}},
$$
with $\eps\in(0,\eps(N))$. Then there exists $\sigma_0=\sigma_0(E,m)>0$ such that for all $\sigma\in(0,\sigma_0)$ there exists a bi-Lipschitz map $\Phi_\sigma:\R^N\mapsto\R^N$,  with $\Phi_\sigma(x)=x$ for all $|x|\geq R$, such that, setting $\widetilde E=\Phi_\sigma(E)$, one has
$$
|\widetilde E|<m, \quad P(E;\overline{B_R})-P(\widetilde E;\overline{B_R})\geq -2^NN\sigma P(E;\overline{B_R}), \qquad |\widetilde E|-|E|\geq c(N)\sigma R^N.
$$
\end{lemma}

\noindent The proof of Lemma \ref{lm:ef} can be deduced from the proof of \cite[Theorem 1.1]{EF}.

\medskip

We conclude this section with the relative isoperimetric inequality outside convex sets due to Choe, Ghomi and Ritor\'e together with the characterization of the equality case. 
\begin{theorem}\label{th:isoperim}
Let $\C\subset\R^N$ be a convex body. For any set of finite perimeter $E\subset\R^N\setminus \C$ we have
\beq\label{isoperim1}
P(E;\R^N\setminus \C)\geq N\Big(\frac{\omega_N}{2}\Big)^{\frac{1}{N}}|E|^{\frac{N-1}{N}}\,.
\eeq
Moreover, if equality holds in \eqref{isoperim1}, then $E$ is a half ball supported on a facet of $\C$.
\end{theorem}

\noindent The proof of \eqref{isoperim1} and the characterization of the equality case when $\C$ is a $C^2$ convex set has been established in \cite{CGR}, while the characterization of the equality case for general convex sets has been obtained in  \cite{FM}.

\section{Asymptotic dimension of a convex body}\label{section dstar} In this section we discuss various properties of the asymptotic dimension $d^*(\C)$ of a convex body $\C$ introduced in \eqref{asymptotic dimension}. We also recall the definition \eqref{recession cone} of the recession cone $\C_\infty$ of $\C$.

 \begin{remark}[Convex bodies with zero/full asymptotic dimension]\label{dstarzero}
 Note that if $\C$ is a bounded convex body then $d^*(\C)=0$. Conversely, if $d^*(\C)=0$ and we fix $x\in \C$, then $\lambda(\C-x)$ must converge to $\{0\}$ as $\lambda \to 0$. Therefore we have that the recession cone  $\C_\infty=\{0\}$  and thus $\C$ is bounded, see \cite[Theorem 8.4]{R}. Note also that if $d^*(\C)=N$ it is clear from the definition that $\C$ contains open balls of arbitrarily large radius. The converse is also true. Indeed, if  $B_{r_n}(x_n)\subset \C$ and $r_n\to \infty$, by testing the definition with $r_n^{-1/2}(\C-x_n)$ we conclude  $d^*(\C)=N$.

 \end{remark}
  
 \begin{remark}  Note that if $\C$ is an unbounded convex body,
 testing \eqref{asymptotic dimension} with the constant sequence $x_n=x\in \C$ and any $\lambda_n\to0$, we deduce that 
 $$
 d^*(\C)\geq\dim \C_\infty,
 $$
with the inequality being possibly strict. Note also that $d^*(\C)$ can also take any value between $1$ and $N$. Indeed,
if $k\in\{1,\dots,N-1\}$, the set $\C=\R^k\times[0,1]^{N-k}$ is such that $d^*(\C)=\dim \C_\infty=k$ (the case $d^*(\C)=N$ being covered in the previous remark). Conversely, Proposition~\ref{general structure lemma} below shows that any convex $\C$ such that $d^*(\C)=k\in\{1,\dots,N-1\}$ is contained in the product of a $k$-dimensional subspace times a bounded convex set. Finally, note that the set $\C=\{(x',x_N)\in\R^{N-1}\times\R:\, x_N\geq|x'|^2\}$ satisfies $\dim \C_\infty=1$, $d^*(\C)=N$.
 \end{remark}

In the next proposition we deal with the structure of a convex body $\C$ such that $1\leq d^*(\C)\leq N-1$. To this aim, if $Z\subset\R^N$ is a subspace, we denote by $\mathbf{p}_{Z}$ the orthogonal projection on $Z$. Furthermore, $Z^\perp$ will denote the subspace orthogonal to $Z$ and if  $z\in\R^N$, we will denote by $z^\perp$ the hyperplane orthogonal to $z$ through the origin.

\begin{proposition}\label{general structure lemma}
Let $\C\subset \mathbb{R}^N$ be a  convex body such that  $1\leq d^*(\C)\leq N-1$. Then there exists a $d^*(\C)$-dimensional subspace $Z$ containing $\C_\infty$ such that
\begin{enumerate}
     \item[(a)] $\C \subset Z + \mathrm{cl}\,(\mathbf{p}_{Z^\perp}(\C))$,
    \item[(b)] $\mathrm{cl}\,(\mathbf{p}_{Z^\perp}(\C))$ is bounded.
  \end{enumerate}
   Moreover, for every $z\in \C_\infty$
\begin{enumerate}
    \item[(c)]  the sets $\C_t:=\mathbf{p}_{z^\perp}(\C\cap (tz+z^\perp))$ satisfy $\C_t \subset \C_{t'}$ whenever $t< t'$,
    \item[(d)]  $\C_t$ converge in the Kuratowski sense to $\mathrm{cl}\,( \mathbf{p}_{z^\perp}(\C))$ as $t\to \infty$, and
    \item[(e)]  $d^*(\mathrm{cl}\,( \mathbf{p}_{z^\perp}(\C)))=d^*(\C)-1$.
   \end{enumerate}
\end{proposition}
\begin{remark} Let $\C$ be an unbounded convex body and assume that $Z$ is a subspace such that $(a)$ and $(b)$ of Proposition~\ref{general structure lemma} hold. Then  $\dim Z=d^*(Z + \mathrm{cl}\,(\mathbf{p}_{Z^\perp}(\C)))\geq d^*(\C)$, where we used the fact that $d^*(\cdot)$ is monotone with respect to the inclusion. Therefore one could equivalently define $d^*(\C)$ as the smallest $k\in\N$ such that there exists a subspace $Z$ with $\dim Z=k$ such that $\mathbf{p}_{Z^\perp}(\C))$ is bounded and $\C \subset Z + \mathrm{cl}\,(\mathbf{p}_{Z^\perp}(\C))$; if no such subspace exists, then $d^*(\C)=N$.
\end{remark}
\begin{proof}[Proof of Proposition~\ref{general structure lemma}] We first prove $(c)$--$(e)$ and then $(a)$ and $(b)$. By Remark~\ref{dstarzero}, $\C$ is unbounded and therefore $\C_\infty$ is nonempty, and so we can fix some $z\in \C_\infty$. By the definition of $\C_\infty$, $x+\alpha z \in \C$ for all $x\in \C$ and $\alpha \geq 0$. From this we deduce the nested property $(c)$ of $\{\C_t\}$. Item $(d)$ then follows from the fact that any sequence of increasing, closed, convex sets converges in the Kuratowski sense to the closure of the union of its elements and observing that the closure of such a union coincides with $\mathrm{cl}\,(\mathbf{p}_{Z^\perp}(\C))$.
\par\medskip
Moving on to $(e)$, let us set $D=\mathrm{cl}\,(\mathbf{p}_{z^\perp}(\C))$ for brevity. We begin by proving that
\begin{align}\label{dstar D leq}
    d^*(D)\leq d^*(\C)-1\,.
\end{align}
Let  $x_n\in  D$ and $\lambda_n \to 0$ be such that $\lambda_n(D - x_n)$ converges in the Kuratowski sense to $D'$, where $\mathrm{dim}\, D'=d^*(D)$. Then, taking into account $0\in D'$, there exist linearly independent vectors $y_1,\dots, y_{d^*(D)}\in D'$. 
By the convergence of $\lambda_n(D - x_n)$ to $D'$, we can therefore choose a sequence $\varepsilon_n \to 0$ and $w_{i,n}\in z^\perp$ such that 
\begin{align}\notag
    w_{i,n} \in B_{\varepsilon_n}(y_i) \cap \lambda_n(\mathrm{ri}(D) - x_n)\, \quad\forall n\in \mathbb{N}\,, 1\leq i \leq d^*(D) \,,
\end{align}
where $\mathrm{ri}(\cdot)$ stands for the relative interior of its argument.
Next, since $\mathrm{ri}(D)\subset \mathbf{p}_{z^\perp}(\C)$ and $z\in \C_\infty$, we may choose $T_n\geq 0$ large enough and $w_{0,n}\in z^\perp$ such that  for all $t\geq T_n$
\begin{align}\notag
    w_{0,n} &\in \lambda_n(\C_t - x_n)\,,\quad w_{0,n}\to 0=:y_0\,, \\ \notag
    w_{i,n} &\in \lambda_n (\C_t - x_n)\,,\quad\,  1\leq i \leq d^*(D)\,,
\end{align}
or, equivalently,
\begin{align}\notag
    w_{0,n} &\in \lambda_n(tz + \C_t - (tz + x_n))\,,\quad w_{0,n}\to 0\,,\\ \notag
    w_{i,n}&\in \lambda_n(tz + \C_t - (tz +x_n)) \subset \lambda_n( \C - (tz + x_n))\,.
\end{align}
Let $T_n\leq t_{1,n}<t_{2,n}$ be such that $\lambda_n(t_{2,n} - t_{1,n})= 1$. Then 
\begin{align}\notag
    w_{0,n} &\in \lambda_n(t_{1,n}z + \C_{t_{1,n}}-(t_{1,n}z + x_n) ) \\ \notag
    w_{1,n}&\in \lambda_n(t_{1,n}z + \C_{t_{1,n}} - (t_{1,n}z + x_n))\subset \lambda_n(\C - (t_{1,n}z + x_n))\,, \\ \notag
    z + w_{i,n}&\in \lambda_n(t_{2,n}z + \C_{t_{1,n}} - (t_{1,n}z + x_n))\subset \lambda_n(\C - (t_{1,n}z + x_n))\quad\forall 1\leq i \leq d^*(D)\,,
\end{align}
where in the last inclusion we used the fact that $\C_{t_{1,n}}\subset \C_{t_{2,n}}$.
By the convergence $w_{i,n}\to y_i$, we can restrict to a further subsequence, which we do not notate, that satisfies $\lambda_n(\C - (t_{1,n}z + x_n)) \to \C'$ in the Kuratowski sense and
\begin{equation}\notag
V:=\{0, y_1, z+y_1, \dots, z+y_{d^*(D)} \}\subset \C'
\end{equation}
for some closed, convex $\C'$. Since $\{y_i\}_{i=1}^{d^*(D)}\subset z^\perp$ are linearly independent, the elements of $V$ are the vertices of a $(d^*(D)+1)$-dimensional simplex $S$, which belongs to $\C'$ since $\C'$ is convex. Since $t_{1,n}z+x_n\in \C$, we have thus shown \eqref{dstar D leq}.\par\medskip
For the reverse inequality
\begin{align}\label{dstar D geq}
    d^*(D)\geq d^*(\C)-1 \,,
\end{align}
we recall that by the definition of $D$,
\begin{align}\label{cylindrical containment}
    \C \subset\mathrm{cl}\,(\mathbf{p}_{z^\perp}(\C))+ \mathrm{span}\,\{z \} =D + \mathrm{span}\,\{z \}\,.
\end{align}
Let $\lambda_n\to 0$, $x_n\in \C$ be such that $\lambda_n(\C-x_n) \to \C'$ for some $\C'$ with $\dim \C'=d^*(\C)$. By \eqref{cylindrical containment}, we have
\begin{align}\notag
    \lambda_n(\C- x_n) &\subset \lambda_n( D + \mathrm{span}\,\{z \} - x_n) \\ \notag
    &=\mathrm{span}\,\{z \}+ \lambda_n ( D   - \mathbf{p}_{z^\perp}(x_n))\,.
\end{align}
Up to a further subsequence, the right hand side of the previous equation converges in the Kuratowski sense to $\mathrm{span}\,\{z \}+D'$ for some closed convex $D'\subset z^\perp$ with $\dim D' \leq d^*(D)$. Thus
\begin{align}\notag
    \dim \C' \leq \dim D' + 1 \leq d^*(D)+1\,,
\end{align}
which is \eqref{dstar D geq}, thus showing item $(e)$.
\par\medskip
To finish the proof  it remains to show $(a)$ and $(b)$, which we do by strong induction on $1\leq d^*(\C)\leq N-1$. For the base case when $d^*(\C)=1$, item $(a)$ follows from  to the containment \eqref{cylindrical containment}. Moreover,
item $(e)$ implies that the associated $D:=\mathrm{cl}\,( \mathbf{p}_{z^\perp}(\C))\subset z^\perp$ satisfies $d^*(D)=0$, and so by Remark~\ref{dstarzero}, $D$ is a bounded convex body contained in $z^\perp$, which is $(b)$.  
\par
\medskip
Suppose now that $(a)$ and $(b)$ are true for any  convex body $\C'$ with $1\leq d^*(\C') \leq k\leq N-2$. If $\C\subset \mathbb{R}^N$ is a convex body with $d^*(\C)=:k+1>1$ and $z\in \C_\infty$, we apply $(c)$, $(d)$ and $(e)$ to $\C$ and to the  $(N-1)$-dimensional convex body $D:=\mathrm{cl}\,( \mathbf{p}_{z^\perp}(\C))\subset z^\perp$ with $d^*(D)=k\geq 1$ to obtain
\begin{equation}\label{C containment}
    \C \subset \mathrm{span}\,\{z \}+ D\,.
\end{equation}
By the induction hypothesis applied to $D$, we may obtain $N-1$ orthonormal vectors $$
\{z_2,\dots,z_{1+d^*(D)},y_{1},\dots, y_{N-d^*(D)-1} \}\subset z^\perp
$$
such that setting $Z'=\mathrm{span}\,\{z_2,\dots,z_{1+d^*(D)} \}\subset z^\perp$ and $Y=\mathrm{span}\,\{y_1,\dots,y_{N-d^*(D)-1} \}\subset z^\perp$, we have 
\begin{equation}\label{D containment}
    D\subset Z'+ \mathrm{cl}\,(\mathbf{p}_{Y}(D))
\end{equation}
and
\begin{equation}\label{D boundedness}
    \mathrm{cl}\,(\mathbf{p}_Y(D))\textup{ is bounded}\,.
\end{equation}
Therfore, by \eqref{C containment}-\eqref{D containment}, setting $Z:= \mathrm{span}\,\{z,z_2,\dots  z_{1+d^*(D)} \}$, we have
\begin{equation}\notag
    \C \subset Z+\mathrm{cl}\,(\mathbf{p}_{Y}(D))= Z+ \mathrm{cl}\,(\mathbf{p}_{Y}(\mathrm{cl}\,(\mathbf{p}_{z^\perp}(\C))))\,.
\end{equation}
Since $d^*(\C)=d^*(D)+1$ and $\mathrm{cl}\,(\mathbf{p}_{Y}(\mathrm{cl}\,(\mathbf{p}_{z^\perp}(\C))))=\mathrm{cl}\,(\mathbf{p}_{Y}(\C))$, we have shown $(a)$. The boundedness of $\mathrm{cl}\,(\mathbf{p}_{Y}(\C))$, which is $(b)$, follows from the induction hypothesis \eqref{D boundedness}.
\end{proof}\par\medskip

\section{Rigidity in the Choe--Ghomi--Ritor\'e comparison theorem}\label{section large volume}
In this section we prove various properties of the exterior isoperimetric profile $I_\C$ of a convex body, and then prove Theorem \ref{theorem main 1}. We begin with the following proposition, where we consider a constrained version of the exterior isoperimetric profile.

\begin{proposition}\label{prop:ICR}
 Let $\C\subset\R^N$ be a convex body. Given $R>0$ sufficiently large, for all $v\in(0, |B_R\setminus \C|)$ we set
\begin{equation}\label{ICR1}
I_{\C, R}(v):=\min\{P(E; \R^N\setminus \C): E\subset B_R\setminus \C\text{ and }|E|=v\}\,.
\end{equation}
Then $I_{\C,R}$ is a locally Lipschitz function in $(0,|B_R\setminus \C|)$ and for a.e. $v>0$, $I'_{\C,R}(v)$ coincides with the constant mean curvature $H_{\pa^*E_v}$ of $\pa^*E_v\cap (B_R\setminus \C)$, where $E_v$ is any minimizer of the problem \eqref{ICR1}. Moreover for all $v>0$
\begin{equation}\label{ICR2}
\lim_{R\to\infty}I_{\C, R}(v)=I_\C(v)\,.
\end{equation}
Furthermore, assuming, up to a translation, that $0\in\pa \C$, there exist  positive constants  $\Lambda_0$, $d_0$, $r_0$, $c_0$, and  an integer $I_0\in\N$,    
    all depending on $N$ (but not on  $\C$), such that for all $R \geq R_0= \big(\frac{2}{\omega_N})^{\frac1N}+1$ and $v>0$, any minimizer $E$ for $I_{\C,R\,v^{1/N}}(v)$ satisfies 
\begin{equation}\label{lambda min eq}
P(E;\mathbb{R}^N\setminus \C) \leq P(F;\mathbb{R}^N\setminus \C)+\Lambda_0v^{-\frac{1}{N}}\big||F|-|E|\big|\quad \forall F\subset B_{R\,v^{1/N}}\setminus \C\,,
\end{equation}
has at most $I_0$ connected components each of them with diameter less than $v^{\frac1N}d_0$ and for every connected component $E'$ and $x\in E'$ 
$$
|E'\cap B_r(x)|\geq c_0r^N \quad \text{for all $0<r\leq r_0v^{\frac1N}$.}
$$
Finally, for any such minimizer we have 
\beq\label{stimacurv}
|H_{\pa^*E}(x)|\leq\Lambda_0v^{-\frac{1}{N}} \quad
 \text{for all $x\in (B_{R\,v^{1/N}}\setminus \C)\cap \pa^*E$.}
 \eeq

\end{proposition} 
\begin{proof} Note that if $R\geq R_0$ for any convex set $\C$ such that $0\in\pa \C$ there exists a set of finite perimeter $E\subset B_R\setminus \C$ with $|E|=1$. Indeed, $\C$ is contained in a half space and $R_0$ is strictly bigger than the radius of a half ball of volume $1$. We divide the proof in three steps.
\par
\medskip
\noindent\textit{Step 1:} The Lipschitz continuity and the representation formula for the derivative of $I_{\C}$ when $\C$ is a bounded convex body are well known facts. The proof of the same properties for $I_{\C,R}$ is similar, see for instance Steps 3 and 4 of the proof of Theorem~1.2 in \cite{FM}. 

The proof of \eqref{ICR2} is immediate.  Indeed we clearly have $I_{\C, R}(v)\geq I_\C(v)$ for $R\geq R_0v^{\frac1N}$. For the opposite inequality it is enough to observe that any competitor for $I_\C(v)$ can be approximated in (relative) perimeter by bounded sets of the same volume.
\par\medskip\noindent
\textit{Step 2:} The argument needed to prove that there exists $\Lambda_0$ such that \eqref{lambda min eq} holds is similar for instance to the one in Step 6 of the proof of \cite[Theorem~3.2]{FFLM} with some modifications due to our particular setting. We reproduce it here for the reader's convenience. To this aim, by rescaling, it is enough to show that there exists $\Lambda_0$ such that for every convex body $\C$, any $v>0$ and any $R\geq R_0$, every minimizer for the penalized problem

\begin{align}\label{penalized problem}
 \min \{P(E;\mathbb{R}^N\setminus (v^{-\frac{1}{N}}\C)) + \Lambda_0\big||E| - 1\big|:E \subset  B_{R}\setminus (v^{-\frac{1}{N}} \C) \}
\end{align}
with volume $1$. 
\par
\medskip
Let us suppose then for contradiction that there exist a sequence $\Lambda_j\to \infty$, $R_j\geq R_0$, $\C_j$, $v_j\in(0,+\infty)$ and minimizers $E_{j,\Lambda_j}$ for \eqref{penalized problem} (with $\Lambda_0$, $R$, $v$ and $\C$ replaced by $\Lambda_j$, $R_j$, $v_j$ and $\C_j$, respectively) which do not have volume $1$. We observe that necessarily $|E_{j,\Lambda_j}|< 1$, since otherwise we could contradict the minimality by cutting $E_{j,\Lambda_j}$ with a hyperplane not intersecting $v_j^{-\frac{1}{N}} \C_j$. Using as a competitor   $B_{\overline R}\setminus (v_j^{-\frac{1}{N}} \C_j)$, with $\overline R\in(0,R_0]$ chosen so that $|B_{\overline R}\setminus (v_j^{-\frac{1}{N}} \C_j)|=1$ (which is possible since $R_j\geq R_0$), we have 
  \begin{equation}\label{perbound}
 P(E_{j,\Lambda_j};\mathbb{R}^N\setminus (v_j^{-\frac{1}{N}}\C_j)))\leq  P(B_{\overline R};\mathbb{R}^N\setminus (v_j^{-\frac{1}{N}}\C_j)))\leq P(B_{R_0})
 \end{equation}
 and
  \begin{equation}\label{volbound}
  |E_{j,\Lambda_j}|\to1\,.
 \end{equation} 
Thus by Lemma~\ref{lm:nonvanishing}, there exists a constant $c(N)>0$ such that $|(z_j+Q)\cap E_{j,\Lambda_j}|\geq c(N)>0$ for some $z_j\in \mathbb{Z}^N$ and for every $j$. Therefore, up to a subsequence (not relabelled), we may assume that  $\Chi{E_{j,\Lambda_j}-z_j}\to \Chi{E}$ a.e., with $E$ of finite perimeter and $|E|\geq c(N)$. 

\medskip

 We claim that there exist $\overline x\in\partial^{\ast}E$ and $\overline r>0$ such that
  \begin{equation}
  	B_{\overline r}(\overline x)\subset -z_{j}+ \big(B_{R_j}\setminus v_j^{-\frac{1}{N}}\C_j\big), \quad \text{for all $j$ sufficiently large.} \label{empty intersection}
  \end{equation}
To see this note  that, up to a not relabelled subsequence, we may assume that $K_j:=-z_j+v_j^{-\frac{1}{N}} \C_j\to K$ in the sense of Kuratowski, for a suitable closed convex set $K$. Moreover,  by Lemma~\ref{Kuratowski equivalence} we have that $\Chi{K_j}\to\Chi{K}$ almost everywhere. In particular, for a.e. $x\in\R^N$ we have  
$\Chi{E}(x)\Chi{K}(x)=\lim_{j}\Chi{E_{j,\Lambda_j}-z_j}(x)\Chi{K_j}(x)=0$, i.e., $E\subset\R^N\setminus K$. Observe also that, up to a further not relabelled subsequence, we may assume that $-z_j+\overline{B_{R_j}}$ converge in the Kuratowski sense to $\widetilde K$, where $\widetilde K$ can be  $\R^N$ or a half space containing $E$ (if $R_j\to\infty$) or a ball $B_{R_\infty}(z_\infty)$, where $R_\infty=\lim_jR_j$ and $z_\infty=-\lim_jz_j$ (if $\{R_j\}$ is bounded). Note that $z_\infty\in\pa K$.
Therefore, there exist $\overline 
x\in\partial^{\ast}E\setminus K$ and $\overline r>0$ such that $\overline{B_{\overline r}({\overline x})}\cap K=\emptyset$ and $\overline{B_{\overline r}({\overline x})}\subset{\rm int}(\widetilde K)$. This is immediate if $\widetilde K$ is either $\R^N$ or a  half space containing $E$. Instead, if $\widetilde K=B_{R_\infty}(z_\infty)$, this follows since $R_\infty\geq R_0$, hence $|B_{R_\infty}(z_\infty)\setminus K|>1$, while $|E|\leq1$ by \eqref{volbound}. Thus, from the Kuratowski convergence   \eqref{empty intersection}  follows.
  
\medskip 
  
 Arguing as in Step~1 of Theorem~1.1 in  \cite{EF}, given $0<\eps<\eps(N)$, where $\eps(N)$ is as in Lemma~\ref{lm:ef},
 we can find a ball $B_r(x_0)\subset B_{\overline r}(\overline x)$ such that 
 $$
 | E\cap B_{\frac{r}{2}}(x_0)|<\e r^N\,, \quad |  E \cap B_r(x_0)|>\frac{\omega_N}{2^{N+2}}r^N\,. 
 $$ 
Therefore, for $j$ sufficiently large, we  have
 $$
 | E_{j,\Lambda_j}\cap B_{\frac{r}{2}}(x_0+z_j)|<\e r^N\,, \quad | E_{j,\Lambda_j}\cap B_{r}(x_0+z_j)|>\frac{\omega_N}{2^{N+2}}r^N\,,
 $$
 where by \eqref{empty intersection}, $B_{r}(x_0+z_j)\subset B_{R_j}\setminus (v_j^{-\frac{1}{N}} \C_j)$. 
 We now apply Lemma~\ref{lm:ef} to find a positive sequence $\{\sigma_j\}$ and a sequence $\{\widetilde E_j\}$ such that $\widetilde E_j\setminus B_{r}(x_0+z_j)=E_{j,\Lambda_j}\setminus B_{r}(x_0+z_j)$ and satisfying $|\widetilde E_j|<1$ and
\begin{align*}
 P(E_{j,\Lambda_j};\overline{B_{r}(x_0+z_j)})-P(\widetilde E_j;\overline{B_{r}(x_0+z_j)})&\geq -2^NN\sigma_j P(E_{j,\Lambda_j};\overline{B_{r}(x_0+z_j)}),\\
 |\widetilde E_j|-|E_{j,\Lambda_j}|&\geq c(N)\sigma_j r^N.
\end{align*}
From these inequalities, recalling \eqref{perbound} and that $|E_{j,\Lambda_j}|<|E_j|<1$,  we then get
\begin{align*}
& P(\widetilde E_j;\R^N\setminus (v_j^{-\frac{1}{N}} \C_j))+\Lambda_j\big||\widetilde E_j|-1\big|- \big(P(E_{j,\Lambda_j};\R^N\setminus (v_j^{-\frac{1}{N}} \C_j))+\Lambda_j\big||E_{j,\Lambda_j}|-1\big|\big) \\
&\quad \leq 2^NN\sigma_j P(E_{j,\Lambda_j};\overline{B_{r}(x_0+z_j)})+\Lambda_j\big(|E_{j,\Lambda_j}|-|\widetilde E_j|\big) \\
&\quad \leq \sigma_j\big(2^NN^2\omega_N^{\frac1N}-\Lambda_jc(N)r^N\big)<0
\end{align*}
for $j$ large, as $\Lambda_j\to\infty$. This contradicts the minimality of $E_{j,\Lambda_j}$, thus proving \eqref{lambda min eq}.
\par\medskip\noindent
\textit{Step 3:} We finally show the last part of the statement. Assume now $R\geq R_0$ and $v>0$ and let $E$ a minimizer for the problem defining $I_{\C,R\,v^{1/N}}(v)$. From \eqref{lambda min eq} it follows that 
$$
P(v^{-\frac1N}E; \R^N\setminus v^{-\frac1N}\C)\leq P(F; \R^N\setminus v^{-\frac1N}\C)+\Lambda_0|F\Delta v^{-\frac1N}E| \qquad\text{for all $F\subset B_R\setminus v^{-\frac1N}\C$.}
$$
Appealing to Lemma \ref{lm:density}, the conclusions of which do not depend on $R$ and $v^{-1/N}\C$, we obtain $c_0(N)>0$ and $r_0$ depending only on $N,\Lambda_0$ and thus only on $N$,  such that for every connected component $E'$ of the open set $v^{-1/N}E$ and $x\in \partial E'$,
\begin{align}\label{density bound exis theorem}
    |E' \cap B_r(x)| \geq c_0r^N\quad\forall 0<r\leq r_0\,.
\end{align}
Since $|v^{-1/N}E|=1$, this implies the existence of $I_0\in\mathbb{N}$ such that the number of connected components of $v^{-1/N}E$ is at most $I_0$. Moreover, the density estimate \eqref{density bound exis theorem} also implies by a standard argument  the existence of $d_0$ such that $\mathrm{diam}\, (E')\leq d_0$ for any  component $E'$. Finally, the estimate \eqref{stimacurv} follows from the $\Lambda$-minimality property \eqref{lambda min eq} by a standard first
variation argument.
\end{proof}

We now prove a semicontinuity property of the exterior isoperimetric profile with respect to the Kuratowski convergence of the set $\C$.

 \begin{lemma}\label{etto}
Let $\{\C_n\}$ be a sequence of convex bodies converging in the Kuratowski sense to a convex body $\C$. Then for all $v>0$ we have
$$
\limsup_n I_{\C_n}(v)\leq I_\C(v)\,.
$$
\end{lemma}
\begin{proof}
Without loss of generality we may assume that $0\in {\rm int}(\C)$. Let $I_{\C,R}$ be the constrained exterior isoperimetric profile defined in \eqref{ICR1}. For any $\e>0$ sufficiently small let $E\subset B_R\setminus \C$ be a minimizer of  the problem defining $I_{\C, R}((1+\e)^Nv)$.  By the local Hausdorff convergence we have that 
$$
\C\cap B_R\subset (1+\e)\C_n\cap B_R\,.
$$
for $n$ large enough.
Set $E_n:=E\setminus (1+\e)\C_n$ and observe that 
\begin{align}\notag
|E_n|\to |E|-\big|E \cap (1+\e)\C\big| = (1+\e)^Nv-\big|E \cap (1+\e)\C\big|\,.
\end{align}
We now let $F_n:=E_n\cup B_{r_n}(x_n)$, with $B_{r_n}(x_n)$ a small ball in $\R^N\setminus (\C_n\cup B_R)$ such that $|F_n|=(1+\e)^Nv$. Then $|B_{r_n}(x_n)| = \mathrm{O}(\e)$ and thus
\[
\begin{split}
\limsup_n I_{(1+\e)\C_n}((1+\e)^Nv)&\leq \limsup_n P(F_n; \R^N\setminus (1+\e)\C_n ) \\
&= \limsup_n P(E_n; \R^N\setminus (1+\e)\C_n ) +\mathrm{O}(\e^{(N-1)/N})\\
&\leq  P(E; \R^N\setminus \C )+\mathrm{O}(\e^{(N-1)/N})= I_{\C,R}((1+\e)^Nv)+\mathrm{O}(\e^{(N-1)/N})\,.
\end{split}
\]
We conclude the proof by observing that  $I_{(1+\e)\C_n}((1+\e)^Nv)= (1+\e)^{N-1}I_{\C_n}(v)$ and using Proposition~\ref{prop:ICR} to let $\e\to 0$ and then $R\to \infty$.
\end{proof}




\begin{lemma}\label{lm:main}
Let $\C\subset\R^N$ be a convex body and let $H$ be any half space. If $\dim \C_\infty\geq N-1$, then $I_\C=I_H$. 
\end{lemma}
\begin{proof}
In what follows we denote by $B'_r(x)$ the intersection $B_r(x)\cap\{x_N=0\}$. We divide the proof in two steps.
\par
\medskip
\noindent\textit{Step 1:} We start with the case $d:=\dim \C_\infty=N-1$.  Without loss of generality we may assume that $\C_\infty\subset\{x_N=0\}$ and that  $e_1$ belongs to the relative interior of $\C_\infty$. Let  $\C'=\C\cap\{x_N=0\}$ and note that $\C_\infty\subset \C'$.  Therefore there exists $r>0$ such that $B'_{2r}(e_1)\subset \C_\infty$. Consider now the sequence of balls $B'_{nr}(ne_1)$ and for any $n$ denote by $\varphi_n: B'_{2nr}(ne_1)\to[0,\infty)$ the concave function whose graph coincides with $\pa \C\cap\big(B'_{2nr}(ne_1)\times[0,\infty)\big)$. We claim that 
\beq\label{lmain1}
\frac{1}{n}\max\big\{\varphi_n(x'):\,x'\in \overline{B'_{nr}(ne_1)}\big\}\to0.
\eeq
In fact, if $x'_n$ is the maximum point of $\varphi_n$ in $\overline{B'_{nr}(ne_1)}$ then, up to a subsequence, 
$$
\frac{(x'_n,\varphi_n(x'_n))}{\sqrt{|x'_n|^2+\varphi_n(x'_n)^2}}\to (y',0)
$$
for some $(y',0)\in \C_\infty$. Thus $\varphi_n(x'_n)/|x'_n|\to0$, hence \eqref{lmain1} follows. Recall that by a well known property of concave functions, see for instance \cite[Ch. I, Eq. (2.15)]{ET},  
$$
\|\nabla\varphi_n\|_{L^\infty(B'_{n\frac{r}{2}}(ne_1))}
\leq \frac{2}{nr}\operatorname*{osc}_{B'_{nr}(ne_1)}\varphi_n\,.
$$
Thus, thanks to \eqref{lmain1} we have that 
\beq\label{lmain1bis}
\|\nabla\varphi_n\|_{L^\infty(B'_{n\frac{r}{2}}(ne_1))}\to0\,.
\eeq
Therefore it is easily checked that the sets $\C-(ne_1,\varphi_n(ne_1))$ converge in the Kuratowski sense, up to a subsequence, to a convex set $K\subset\{x_N\leq0\}$ such that $\pa K\subset\{x_N=0\}$. Let us now fix $v>0$ and denote by $r_v$ the radius of a half ball of volume $v$. For every $n$ let $r_n$ be the radius of the ball centered at $x_n=(ne_1,\varphi_n(ne_1))$ and such that $|B_{r_n}(x_n)\setminus \C|=v$. Then, recalling that by \eqref{lmain1bis} the boundary of $\C-x_n$ is flattening out,  it follows that $r_n\to r_v$ and that $P(B_{r_n}(x_n);\R^N\setminus \C)\to I_H(v)$. Hence, $I_\C(v)\leq I_H(v)$, while the opposite inequality follows from Theorem \ref{th:isoperim}.

\medskip

\noindent\textit{Step 2:} We assume now that $d=N$. Without loss of generality we may assume, up to a possible rotation and dilation that $\C_\infty$ has a unique tangent plane $\{x_N=0\}$ at $e_1$ and that $\C_\infty$ stays above $\{x_N=0\}$. Let us now fix $\kappa>0$ so large that $B'_{\frac{1}{\kappa}}(e_1)$ is contained in the projection of $\C_\infty$ onto $\{x_N=0\}$.  Consider the balls $B'_{\frac{n}{\kappa}}(ne_1)$ and the functions $\varphi_n:\overline{B'_{\frac{n}{\kappa}}(ne_1)}\to [-\infty,+\infty)$ defined as 
$$
\varphi_n(x')=\inf\{t\in\R:\,\,(x',t)\in\pa \C\}.
$$
Observe that for any $x'\in B'_{\frac{n}{\kappa}}(ne_1)$ the above infimum is finite. Indeed, if for some point $x'$ we had $\varphi_n(x')=-\infty$, then the half line $\{te_N:\,t\leq0\}$ would be contained in $\C_\infty$, which is not possible. A similar argument shows also that if $x'_n$ is the minimum point of $\varphi_n$ on $\overline{B'_{\frac{n}{\kappa}}(ne_1)}$ then 
$
\liminf_{n}\frac{\varphi_n(x'_n)}{n}>-\infty.
$
We claim that 
\beq\label{lmain2}
 \liminf_{n}\frac{\varphi_n(x'_n)}{n}\geq0.
 \eeq 
  Indeed, 
 assuming without loss of generality that the sequence $\big(\frac{x'_n}{n},\frac{\varphi_n(x'_n)}{n}\big)$ converges to some point $y=(y',y_N)\in\R^n$ then necessarily $y\in\pa \C_\infty$ and thus $y_N\geq0$. Since $\varphi_n$ is a convex function, as before we
 have
\beq\label{lmain3}
 \|\nabla\varphi_n\|_{L^\infty(B'_{\frac{n}{2\kappa}}(ne_1))}
\leq \frac{2\kappa}{n}\operatorname*{osc}_{B'_{\frac{n}{\kappa}}(ne_1)}\varphi_n=\frac{2\kappa}{n}(\varphi_n(y'_n)-\varphi_n(x'_n)),
 \eeq
 where $y'_n$ is the maximum point of $\varphi_n$ on $\overline{B'_{\frac{n}{\kappa}}(ne_1)}$. Observe that  the point $\big(\frac{y'_n}{n},\frac{\varphi_n(y'_n)}{n}\big)$ lies below $\pa \C_\infty$ and since $\frac{y'_n}{n}\in B'_{\frac{1}{\kappa}}(e_1)$ we have
 $$
 \frac{\varphi_n(y'_n)}{n}\leq \sup\big\{t:\, (x',t)\in \pa \C_\infty\,\, \text{with $x'\in B'_{\frac{1}{\kappa}}(e_1)$}\big\}\leq o\Big(\frac{1}{\kappa}\Big),
 $$
 where the last inequality follows from the fact $\{x_N=0\}$ is tangent to $\C_\infty$ at $e_1$.
 Therefore, recalling \eqref{lmain2} and \eqref{lmain3} we get that 
 $$
 \|\nabla\varphi_n\|_{L^\infty(B'_{\frac{n}{2\kappa}}(ne_1))}
\leq \kappa o\Big(\frac{1}{\kappa}\Big).
 $$
 Fix $R>0$. From the previous estimate we have  that for all $\kappa>0$
 $$
 \limsup_n \|\nabla\varphi_n\|_{L^\infty(B'_{R}(ne_1))}\leq\limsup_{n} \|\nabla\varphi_n\|_{L^\infty(B'_{\frac{n}{2\kappa}}(ne_1))}\leq\kappa o\Big(\frac{1}{\kappa}\Big),
 $$
 hence for all $R>0$
 $$ \lim_n \|\nabla\varphi_n\|_{L^\infty(B'_{R}(ne_1))}=0.$$
 Then conclusion then follows exactly as in the final part of Step 1.
\end{proof}

We are finally ready to prove Theorem \ref{theorem main 1}.

\begin{proof}[Proof of Theorem \ref{theorem main 1}] It suffices to prove that if $d^*(\C)\geq N-1$, then $I_\C\equiv I_H$, and that if $d^*(\C)\le N-2$, then $I_\C(v)\,v^{(1-N)/N}\to N\omega_N^{1/N}$ as $v\to\infty$.

\medskip

If $d^*(\C)\geq N-1$, by definition there exist a sequence $\{x_n\}\subset \C$, $\lambda_n\to0^+$, such that $\lambda_n(\C-x_n)\to K_{\rm max}$ 
  in the Kuratowski sense, with $\dim K_{\rm max}=d^*(\C)\geq N-1$. Up to a subsequence, we may also assume that $\C-x_n\to K$ for some convex body $K$. In particular, since for any $\lambda>0$ and $n$ sufficiently large 
  $$
  \lambda_n(\C-x_n)\subset \lambda(\C-x_n)\to \lambda K, 
  $$ 
  we have that $K_{\rm max}\subset \lambda K$ for every $\lambda>0$, and thus $K_{\rm max}\subset K_\infty$. Therefore $\dim K_\infty\geq N-1$.
 In turn, by Lemma~\ref{lm:main} we have that $I_K=I_H$. Hence, the lower semicontinuity property stated in Lemma~\ref{etto} for all $v>0$ we have that $I_\C(v)\leq I_K(v)=I_H(v)$, while the opposite inequality follows by the isoperimetric inequality \eqref{isoperim1}.

\medskip

Assume now $d^*(\C)\le N-2$. Without loss of generality we may assume $0\in\pa \C$.
Given any diverging sequence $v_n\to+\infty$,  it will be enough to show that 
\begin{equation}\label{claim3}
\liminf_{n}\frac{I_\C(v_n)}{v_n^{\frac{N-1}N}}\geq N\omega_N^{\frac1N}\,.
\end{equation}
Without loss of generality we may assume that the above $\liminf$ is a limit. By Proposition~\ref{prop:ICR}, we may find $\Lambda_0>0$ and $R_n\to+\infty$  such that 
\beq\label{claim3.5}
\lim_{n}\frac{I_{\C, R_n}(v_n)}{v_n^{\frac{N-1}N}}= \lim_{n}\frac{I_\C(v_n)}{v_n^{\frac{N-1}N}}\,,
\eeq
and
\beq\label{claim4}
\begin{split}
\frac{I_{\C, R_n}(v_n)}{v_n^{\frac{N-1}N}}&=I_{v_n^{-\frac1N}\C, v_n^{-\frac1N}R_n}(1)\\
&=\min\Big\{P(E; \R^N\setminus v_n^{-\frac1N}\C)+\Lambda_0\big||E|-1\big|:\, E\subset v_n^{-\frac1N}\big( B_{R_n}\setminus \C\big)  \Big\}\,.
\end{split}
\eeq
Let $E_n$ be a minimizer of \eqref{claim4}. Again by Proposition~\ref{prop:ICR}, passing possibly to a not relabelled subsequence,  there exist $\kappa\in \N$ and $d_0>0$ such that  for $n$ sufficiently large each $E_n$ has  $\kappa$ connected components $E_{n,i}$ and each of them has diameter less than $d_0$. 
We claim that for all $i=1,\dots,\kappa$
\beq\label{claim5}
P(E_{n,i}; \pa (v_n^{-\frac1N}\C) )\to 0.
\eeq
To this aim fix $i$ and assume $P(E_{n,i}; \pa (v_n^{-\frac1N}\C) )\nrightarrow 0$. Then we may find  $x_{n,i}\in v_n^{-\frac{1}{N}}\C$ such that $E_{n,i}\subset B_{d_0}(x_{n,i})$.  Up to a subsequence, we have $v_n^{-\frac{1}{N}}\C-x_{n,i}\to K_i$ with $\dim K_i<N-1$.
Hence, for any  $\e>0$ and $n$ large enough,  setting $(K_i)_\e:=\{x:\,{\rm dist}(x,K_i)\leq\e\}$, we have  
\[
\begin{split}
& P(E_{n,i}-x_{n,i}; \pa (v_n^{-\frac1N}\C-x_{n,i}) )\leq  \H^{N-1}(B_{d_0}\cap  \pa (v_n^{-\frac1N}\C-x_{n,i}) )\\
&\leq  P(B_{d_0}\cap   (v_n^{-\frac1N}\C-x_{n,i}))\leq  P(B_{d_0}\cap   (K_i)_\e)\,,
\end{split}
\]
where the last inequality follows from the containment $B_{d_0}\cap   (v_n^{-\frac1N}\C-x_{n,i})\subset B_{d_0}\cap   (K_i)_\e$ (holding for $n$ large by Kuratowski convergence) and the fact that both sets are convex. 
Since $P(B_{d_0}\cap   (K_i)_\e)\to 0$ as $\e\to0$,  \eqref{claim5} follows.

\medskip

In turn, by the isoperimetric inequality,
$$
\frac{I_{\C, R_n}(v_n)}{v_n^{\frac{N-1}N}}=P(E_n; \R^N\setminus v_n^{-\frac1N}\C)=P(E_n)-P(E_n; \pa (v_n^{-\frac1N}\C))\geq N\omega_N^{\frac1N}-P(E_n; \pa (v_n^{-\frac1N}\C))
$$
and \eqref{claim3} then follows, recalling  \eqref{claim3.5} and \eqref{claim5}. This concludes the proof of the theorem.
\end{proof}
\begin{corollary}\label{nonex}
Let $\C$ be a  strictly convex body with $d^*(\C)\geq N-1$. Then, for every $v>0$ the relative isoperimetric problem \eqref{ICV} does not admit a solution. 
\end{corollary}
\begin{proof}
By Theorem~\ref{theorem main 1}, we have $I_\C(v)=I_H(v)$. On the other hand, by the characterization of the equality case  in Theorem~\ref{th:isoperim}, taking into account that $\C$ is strictly convex, we have $P(E; \R^N\setminus \C)>I_H(v)$ for all sets $E\subset \R^N\setminus \C$, with $|E|=v$.
\end{proof}

\begin{remark}[Minimizers and generalized minimizers]\label{rm:existence}
Note that if $\C$ is a convex cylinder, then there exists a minimizer  for the problem defining  $I_\C(v)$ for  all $v>0$. Indeed, by Proposition~\ref{prop:ICR} any minimizer $E_R$ of $I_{\C,R v^{1/N}}(v)$ for $R\geq R_0$ has at most $I_0$ connected components of diameter at most $d_0 v^{1/N}$. Letting $R\to \infty$ and availing ourselves of the translation invariance outside a cylinder and the convergence of $I_{\C,R\,v^{1/N}}$ to $I_\C(v)$, we see that up to a subsequence and translations, the $E_R$'s converge to a minimizer $E$ of $I_\C(v)$ as $R\to \infty$. The veracity of \eqref{lambda min eq} among any $F\subset\!\subset\mathbb{R}^N\setminus \C$ follows from noticing that  for any $B_R(0)$ containing $E_R\cup F$ \eqref{lambda min eq} is satisfied by $E_R$ and then passing to the limit as $R\to\infty$.

Finally, we remark that  for general convex sets $\C$ one could prove the existence of {\em generalized minimizers} for the problem $I_{\C}(v)$ in the following sense: Let $E_R$ be as before, with $R\to +\infty$, and pick $z_R\in E_R$. Then, up to a subsequence, we may assume that $E_R-z_R\to E_\infty$  in $L^1$ and $\C-z_R\to K_\infty$ in the Kuratowski sense, with $K_\infty$ being a (possibly lower dimensional) convex cylinder, see \cite[Lemma 3.1]{LRV}. 
It could be possible to show that $E_\infty$ is a minimizer for the ``asymptotic problem''  $I_{K_\infty}(v)$. The set $E_\infty$ can be regarded as a generalized minimizer for $I_\C(v)$  capturing the behaviour of  (suitable) minimizing sequences.  Note that $d^*(K_\infty)\leq d^*(\C)$.
\end{remark}

\section{The order of isoperimetric residue for unbounded convex bodies with asymptotic codimension larger than $2$}\label{section residues} The main goal of this section is providing a proof of Theorem \ref{theorem order of residue}. In fact, we shall prove various other results that seem potentially useful for future investigations too. Specifically, as already explained in the introduction and as indicated by \cite{MN} for the case $d^*(\C)=0$, the question of understanding the behavior of $\mathcal{R}_\C(v)$ as $v\to\infty$ is closely related to the description of minimizers of $I_\C(v)$ as $v\to\infty$. Since such minimizers may fail to exist, here we explore the idea of using minimizers of the constrained isoperimetric problems $I_{\C,R\,v^{1/N}}$ already used in the previous section. In particular, in the following lemmas, we obtain basic information on the shape of such minimizers; see, in particular, Theorem \ref{existence and uniform minimality} and Proposition \ref{cylinder hausdorff lemma take 2}.

\medskip

Before moving forward with the above program we introduce some additional notation and terminology. In the following we denote by $c(N,\C)$ a positive constant depending only on $\C$ and $N$ whose value may change from line to line or even within the same line. Moreover, given a set of finite perimeter  $E$, with $v=|E|$, we denote the {\bf isoperimetric deficit} and the {\bf Fraenkel asymmetry} of $E$ by
\beq\label{fraenkel}
\delta_{\mathrm{iso}}(E):=\frac{P(E)}{P(B^{(v)})} -1, \quad \mathcal A(E):=  \inf_{x\in \mathbb{R}^N}\frac{|E\Delta B^{(v)}(x)|}{v},
\eeq
respectively. Finally, in this section, if $A$ and $B$ are positive quantities associated with a fixed  convex body $\C$, by
$$
A  \lesssim B
$$
we mean that there exists a constant $c(N,\C)>0$  such that $A\leq c(N,\C) B$.

\begin{lemma}[Estimate from above by convex cylinders]\label{determination of the isoperimetric profile}
Let $N\geq 3$, $\C\subset \mathbb{R}^N$ be a closed convex body with $1\leq d^*(\C)\leq N-2$, $Z$ be a $d^*(C)$-dimensional subspace as in Proposition~\ref{general structure lemma}, and $D:=\mathrm{cl}\,(\mathbf{p}_{Z^\perp}(\C))$. Then 
\begin{equation}\label{comparable to cylinder profile}
    I_\C(v) \leq I_{Z+ D}(v)\quad\forall v>0\,.
\end{equation}
\end{lemma}

\begin{proof}
Let $z_1\in \C_\infty$. By (c), (d) and (e) of Proposition~\ref{general structure lemma} we have that  $\C-tz_1$ converges in the Kuratowski sense to $\mathrm{span}\,\{z_1\} + D_1$ as $t\to+\infty$, where $D_1=\mathrm{cl}\,(\mathbf{p}_{z_1^\perp}(\C))$ and $d^*(D_1)=d^*(\C)-1$. Then by Lemma~\ref{etto} we have that
$$
I_\C(v)\leq I_{\mathrm{span}\,\{z_1\} + D_1}(v)\,.
$$
If $d^*(\C)=1$ we have proved the claim. Otherwise, we pick any $z_2\in (D_1)_\infty$ and, arguing as above, we get that 
$D_1-tz_2\to \mathrm{span}\,\{z_2\} + D_2$ as $t\to+\infty$ in the Kuratowski sense, so that $\mathrm{span}\,\{z_1\} + D_1-tz_2\to\mathrm{span}\,\{z_1,z_2\} + D_2$, where $D_2=\mathrm{cl}(\mathbf{p}_{z_2^\perp}( D_1))=\mathrm{cl}(\mathbf{p}_{Z_2^\perp}( \C))$ and  $Z_2=\mathrm{span}\,\{z_1,z_2\}$. Now $d^*(D_2)=d^*(\C)-2$ and again, by upper semicontinuity
$$
I_\C(v)\leq I_{Z_2 + D_2}(v)\,.
$$
The conclusion then follows by iterating the argument.
\end{proof}

The following theorem contains Theorem \ref{theorem order of residue} as a particular case.

\begin{theorem}[Lower and upper bounds for $\mathcal{R}_\C$]\label{existence and uniform minimality}
    Let $N\geq 3$ and let $\C\subset \mathbb{R}^N$ be a  convex body with  $1\leq d^*(\C)\leq N-2$. Assume without loss of generality that $0\in\pa \C$. Then, there exists  $v_0\geq1$ 
    depending on $\C$ and $N$, such that for all $R \geq R_0= \big(\frac{2}{\omega_N})^{\frac1N}+1$ and $v\geq v_0$, any minimizer $E$ for $I_{\C,R\,v^{1/N}}(v)$ has one connected component and satisfies {\rm diam}$(E)\leq v^{\frac1N}d_0$, where $d_0$ is as in Proposition~\ref{prop:ICR}.
Moreover,  if $v\geq v_0$
\begin{align}\label{first lower  residue bound}
 v^{\frac{d^*(\C)}{2N}} \lesssim\mathcal{R}_\C(v)\lesssim
     (\mathrm{diam}\,\C \cap \partial E)^{d^*(\C)} \lesssim v^{\frac{d^*(\C)}{N}}\,, 
\end{align}
and
\begin{align}\label{isop deficit}
     \mathcal A(E)^2 \leq c(N)\delta_{\mathrm{iso}}(E)\lesssim&\frac{(\mathrm{diam}\, (\C \cap \partial E))^{d^*(\C)}}{v^{(N-1)/N}}\lesssim v^{-\frac{N-1-d^*(\C)}{N}}\,.
\end{align}
\end{theorem}
\begin{proof} Throughout the proof we will use the fact that, by Proposition~\ref{general structure lemma}, there exists a subspace $Z$  of dimension $d^*(\C)$, such that $\mathrm{cl}\,(\mathbf{p}_{Z^\perp}(\C))$ is bounded and
\begin{align}\label{exis theorem containment}
\C\subset Z+ \mathrm{cl}\,(\mathbf{p}_{Z^\perp}(\C))=:\widetilde \C\,.
\end{align}
  For each $v>0$, let us now choose $R(v)\geq R_0$ such that
\begin{align}\notag
I_{\C,R(v)v^{1/N}}(v)\leq I_\C(v)+1
\end{align}
and denote by $E_v\subset B_{R(v)v^{1/N}}\setminus \C$ a minimizer for $I_{\C,R(v)v^{1/N}}(v)$.
By Proposition~\ref{prop:ICR}
there exist $I_0\in\mathbb{N}$ and $d_0$, depending only on $N$, such that the number of connected components of $E_v$ is at most $I_0$ and
\begin{equation}\label{diameter bound}
    \mathrm{diam}\, (E')\leq d_0v^{\frac1N}\,,
\end{equation}
for any connected component $E'$ of $E_v$. Given $y\in\R^N$ we shall write $y=(y',y^\perp)$, with $y'\in Z$ and $y^\perp\in Z^\perp$. 
 Let $E'$ be a connected component of $E_v$ touching $\pa \C$ and let $x=(x',x^\perp)\in\pa E'\cap \pa \C$.  Note that $E'\subset S=\{y=(y',y^\perp):\,|y'-x'|\leq {\rm diam}\,(E')\}$. Hence, using also \eqref{exis theorem containment}, we have
 \begin{align}\label{tag}
    P(E'; \partial \C) &\leq P(\C\cap S) \leq P(\widetilde \C\cap S)\,.
\end{align}
Note that $\widetilde \C\cap S=B\times \mathrm{cl}\,(\mathbf{p}_{Z^\perp}(\C))$, where $B=\{y'\in Z:\,|y'-x'|\leq {\rm diam}\,(E')\}$.
Therefore, denoting by $\pa_Z$ and $\pa_{Z^\perp}$ the boundary relative to $Z$ and $Z^\perp$, respectively,
\beq \label{4.6}
\begin{split}
 P(\widetilde \C\cap S) 
& \leq c(d^*(\C),N)\Big(\H^{d^*(\C)}(B)\mathcal{H}^{N-d^*(\C)-1}\big(\pa_{Z^\perp} (\mathbf{p}_{Z^\perp}(\C))\big) \\
 & \quad +\H^{d^*(\C)-1}(\pa_Z B)\mathcal{H}^{N-d^*(\C)}\big( (\mathbf{p}_{Z^\perp}(\C)\big)\Big)\\
&\lesssim \big({\rm diam}\,(E')^{d^*(\C)}+{\rm diam}\,(E')^{d^*(\C)-1}\big)\lesssim v^{\frac{d^*(\C)}{N}}\,, 
\end{split}
\eeq
provided that $v\geq1$, where in the last inequality  we used \eqref{diameter bound}. From this estimate, \eqref{tag} and the fact that there are at most $I_0$ connected components, we have $P(E_v; \partial \C)\lesssim v^{\frac{d^*(\C)}{N}}$. In turn, this implies 
\begin{align} 
    N\omega_N^{\frac{1}{N}}v^{\frac{N-1}{N}} &\leq P(E_v) = P(E_v;\mathbb{R}^N\setminus \C) + P(E_v;\partial \C) \nonumber \\
     &\leq I_\C(v)+1+ c(\C,N)v^{\frac{d^*(\C)}{N}} \label{4.7} \\
     &\leq N\omega_N^{\frac{1}{N}}v^{\frac{N-1}{N}} + c(\C,N)v^{\frac{d^*(\C)}{N}}\,. \nonumber
\end{align}
 Now we show that up to increasing the value of $v_0$ obtained in step two, any minimizer $E$ for $I_{\C,R\,v^{1/N}}(v)$ with $R \geq R_0$ has one connected component. By \ the isoperimetric inequality, $v^{-1/N}E=\cup_{i=1}^{I_0} E_i$ satisfies
\begin{align}\label{impossible perimeter bound}
    N\omega_N^{\frac{1}{N}}\sum_{i=1}^{I_0}|E_i|^{\frac{N-1}{N}}\leq \sum_{i=1}^{I_0}P(E_i) \leq N\omega_N^{\frac{1}{N}}+\mathrm{O}\big(v^{-\frac{N-1-d^*(\C)}{N}}\big)\,,
\end{align}
where the $E_i$'s are the connected components of $v^{-1/N}E$, $\sum_{i=1}^{I_0} |E_i|=1$ and the last inequality follows arguing as in  \eqref{4.7}\footnote{Note indeed that \eqref{4.7} holds with  $I_\C(v)+1$ replaced by $I_{\C,R\,v^{1/N}}(v)$ and with $E_v$ any  minimizer for the problem defining $I_{\C,R\,v^{1/N}}(v)$, provided $R\geq R_0$.}. But for all $i$, $|E_i|\geq c_0 r_0^N$, where $r_0$ is as in Proposition~\ref{prop:ICR}, and so by the concavity of $t\mapsto t^{(N-1)/N}$, $I_0\geq 2$ is impossible for large enough $v$. Therefore, recalling \eqref{tag} and \eqref{4.6}, and arguing as in \eqref{4.7} we have
\begin{align*} 
    N\omega_N^{\frac{1}{N}}v^{\frac{N-1}{N}} &\leq P(E_v) = P(E_v;\mathbb{R}^N\setminus \C) + P(E_v;\partial \C) \nonumber \\
     &\leq I_\C(v)+1 +\mathrm{O}\big((\mathrm{diam}\,(\C\cap\partial E))^{d^*(\C)}\big)\,.    
     \end{align*}
Hence,  the last two inequalities in \eqref{first lower  residue bound} follow from the above inequality and  \eqref{diameter bound}. In turn, \eqref{isop deficit} follows from these inequalities and from the quantitative isoperimetric inequality proved in \cite{FMP08}.
\par
\medskip
We now prove the first inequality in \eqref{first lower  residue bound}. By Lemma~\ref{determination of the isoperimetric profile} it suffices to estimate $\mathcal{R}_\C$ from below when $\C$ is a convex cylinder of the form $Z+D$, with $D\subset Z^\perp$ bounded and $Z$ a subspace of dimension $d^*(\C)$. With no loss of generality we may assume that $Z=\{x\in\R^N:\, x=(x_1,\dots,x_{d^*(C)},0\dots,0)\}$ and that
the cube $Q:= [-\alpha,\alpha]^{N-d^*(\C)}\subset \mathrm{ri}\, D$.  
\par\noindent
In the following we will denote a point in $\R^N$ as $x=(x',y',x_N)$, where $x'\in Z$, $y'=(x_{d^*(\C)+1},\dots,x_{N-1})$. 
We will simply attach a large ball to $Z+D$, utilizing $Q$ to bound from below $\mathcal{R}_{Z+D}(v)$. 
For every $r$, consider the ball $B_r(-re_N)$. By the choice of  $Q$, we estimate
\begin{align}
    \mathcal{H}^{N-1}(\pa B_r(-re_N)\cap(Z+D)) &\geq \mathcal{H}^{N-1}\left(B_r(-re_N)\cap \big(Z\times[-\alpha,\alpha]^{N-1-d^*(\C)}\times\{-\alpha\}\big)\right) \nonumber\\ 
    &=\int_{[-\alpha,\alpha]^{N-1-d^*(\C)}}dy'\int_{\{x'\in Z:\, |x'|^2+|y'|^2\leq 2r\alpha-\alpha^2\}}dx'\label{perimeter inside rectangle} \\
    &\geq c(N,d^*(\C))\alpha^{N-1-d^*(\C)}(\alpha r)^{\frac{d^*(\C)}{2}}\nonumber
\end{align}
for $r$ sufficiently large. By a similar argument we may estimate
\begin{equation}\label{volume construction estimate}
  \omega_N r^N - c(N,\C)r^{\frac{d^*(\C)}{2}}\mathrm{diam}\,(D)^{N- d^*(\C)}\leq v_r:=|B_r(-re_N)\setminus \C| \leq \omega_N r^N\,.
\end{equation}
Up to changing the constant $c(N,\C)$ as necessary, we combine \eqref{perimeter inside rectangle} and \eqref{volume construction estimate} to obtain
\begin{align*}
    I_\C(v_r) &\leq P(B_r(-re_N);\mathbb{R}^N\setminus \C)\leq N\omega_Nr^{N-1}-c(N,\C)r^{\frac{d^*(\C)}{2}}\\
    &\leq N\omega_N\left(\frac{v_r}{\omega_N}+c(N,\C)r^{\frac{d^*(\C)}{2}}\right)^{\frac{N-1}{N}}- c(N,\C)r^{\frac{d^*(\C)}{2}}\\ 
    &\leq N\omega_N^{\frac{1}{N}}v_r^{\frac{N-1}{N}}- c(N,\C)v_r^{\frac{d^*(\C)}{2N}}
\end{align*}
for large enough $r$, where all the constants above may change from line to line. Since $\mathcal{R}_\C(v_r) =N\omega_N^{\frac{1}{N}}v_r^{\frac{N-1}{N}}-I_\C(v_r)$, we have proven the lower bound in \eqref{first lower  residue bound}.
 \end{proof}\par\medskip

In the next lemma we complement the upper bound given in \eqref{comparable to cylinder profile} by a corresponding lower bound.

\begin{lemma}[Estimate from below by convex cylinders]\label{determinationbis}
Let $N\geq 3$, $\C\subset \mathbb{R}^N$ be a closed convex body with $1\leq d^*(\C)\leq N-2$, and $Z$ be a $d^*(C)$-dimensional subspace as in Proposition~\ref{general structure lemma} with $D:=\mathrm{cl}\,(\mathbf{p}_{Z^\perp}(\C))$. Then there exist $v_0>0$ and $c(N, \C)$ such that
\begin{equation}\label{comparablebis}
    I_{Z+ D}(v) - c(N,\C)v^{\frac{d^*(\C)-1}{N}} \leq I_\C(v) \quad\forall v\geq v_0\,.
\end{equation}
In particular, 
\begin{equation}\label{seconda}
    0\leq \mathcal{R}_\C(v)-\mathcal{R}_{Z + D}(v) \lesssim v^{\frac{d^*(\C)-1}{N}}\quad\forall v\geq v_0\,.
\end{equation}
\end{lemma}
\begin{proof}
By Theorem \ref{existence and uniform minimality} (and assuming without loss of generality $0\in \pa\C$), given $v\geq v_0$ and $R\geq R_0$ there exists a connected minimizer  $E_v$ of the problem defining $I_{\C,R\,v^{1/N}}(v)$ such that  $\mathrm{diam} \,E_{v} \leq d_0 v^{\frac{1}{N}}$.
By the containment of $\C$ in $Z+ D$, we may write 
\begin{align}\label{lower bound by D step 1}
 P(E_v;\mathbb{R}^N\setminus \C)     &\geq P(E_v;\mathbb{R}^N\setminus (Z+ D)) \,.
\end{align}
Now by the diameter bound on $E_v$ and boundedness of $D$, we know that 
\beq\label{stima volume}
\begin{split}
    |E_v \setminus (Z+ D)| &\geq |E_v \setminus \C| - |E_{v} \cap (Z+ D)|  \\ 
    &\geq v- c(N,\C)(\mathrm{diam} \,E_{v})^{d^*(\C)}(\mathrm{diam} \,D)^{N-d^*(\C)} \\
    &\geq v - cv^{\frac{d^*(\C)}{N}}\geq v_0\,,
\end{split}
\eeq
where $c=c(N,\C)>0$, provided $v$ is sufficiently large.

\medskip

Recall now that by Proposition~\ref{prop:ICR} for a.e. $t\in (v- c v^{\frac{d^*(\C)}{N}},v)$ and for $R$  sufficiently large,  we have that $|I_{Z+D,R\,v^{1/N}}'(t)|\leq \Lambda_0 t^{-\frac 1N}$. Thus, 
\beq\label{stimetta}
    0\leq I_{Z+D,R\,v^{1/N}}(v) - I_{Z+D,R\,v^{1/N}}
    (v-cv^{\frac{d^*(\C)}{N}}) \leq \Lambda_0(N) \int_{v-cv^{\frac{d^*(\C)}{N}}}^vt^{-\frac{1}{N}}\,dt 
    \leq c(N,\C)v^{\frac{d^*(\C)-1}{N}}\,.
\eeq
Combining \eqref{lower bound by D step 1} with \eqref{stimetta}, and taking into account \eqref{stima volume}, we get
\begin{align}\notag
   I_{\C,R\,v^{1/N}}(v) &\geq P(E_v;\mathbb{R}^N\setminus (Z+ D)) \\ \notag
    &\geq I_{Z+ D, R\,v^{1/N}}(v-c v^{\frac{d^*(\C)}{N}}) \\ \notag
    &\geq  I_{Z+ D, R\,v^{1/N}}(v) - c(N,\C)v^{\frac{d^*(\C)-1}{N}}\,.
\end{align}
Letting $R\to\infty$, we conclude \eqref{comparablebis}. Finally, \eqref{seconda} follows at once from \eqref{comparablebis} and \eqref{comparable to cylinder profile}.
\end{proof}
\begin{remark}\label{rm:d*=1}
Note that by the above lemma, if $d^*(\C)=1$, then the behaviour of the residue $\mathcal{R}_{\C}$ is determined, up to a constant, by the residue $\mathcal{R}_{\textup{span\,}z+D}$, $D=\mathrm{cl}(\mathbf{p}_{z^\perp}(\C))$, with $z$ being any point of $\C_\infty$. 
In particular, in the physical case $N=3$,  we can reduce the study of  isoperimetric residues to the case of convex cylinders.  
\end{remark}

We conclude with the following proposition, which provides an estimate on the proximity to balls of large volume isoperimetric sets. For simplicity we assume $\C$ to be a convex cylinder, as in this case we can ensure the existence of minimizers for the relative isoperimetric problem (see Remark~\ref{rm:existence}). Given two compact sets $K_1, K_2$, we denote here by
$$
  \mathrm{hd}\,(K_1, K_2)= \max\big\{\max_{x\in K_1}\mathrm{dist}\,(x,K_2),\max_{x\in K_1}\mathrm{dist}\,(x,K_1)\big\}\,,
$$
their Hausdorff distance.

\begin{proposition}[Uniform convergence to balls]\label{cylinder hausdorff lemma take 2}
Let $1\leq k\leq N-2$  and let $\C=Z+D$, with $Z$ a $k$-dimensional subspace and $D\subset Z^\perp$ a compact $(N-k)$-dimensional convex set.  Then there exists $v_1=v_1(N, \C)>0$ such that if $E$ is a minimizer for the problem defining $I_\C(v)$, with $v\geq v_1$, and  if the ball $B^{(v)}(x_0)$ is optimal for the definition  \eqref{fraenkel} of $\mathcal A(E)$, 
then
\begin{align}\label{hausdorff estimate}
    \frac{\mathrm{hd}\,(\partial E \setminus \C, \partial B^{(v)}(x_0))}{v^{1/N}}\lesssim \frac{(\mathrm{diam}\,\C \cap \partial E)^{k/(2N)}}{v^{(N-1)/(2N^2)}}\lesssim v^{-(N-1-k)/(2N^2)}\,.
\end{align}
\end{proposition}
\begin{remark}
It should be noted that  for general convex sets $\C$, with $1\leq d^*(\C)\leq N-2$,  the conclusion of Proposition~\ref{cylinder hausdorff lemma take 2} applies to the generalized minimizers introduced in  Remark~~\ref{rm:existence}, with $k=d^*(\C)$.
\end{remark}

\begin{proof}[Proof of Proposition~\ref{cylinder hausdorff lemma take 2}]
Let $v\geq v_0$, where $v_0$ is from Theorem \ref{existence and uniform minimality}, so that any minimizer $E$ is connected.
Recall also that by Proposition~\ref{prop:ICR}
\begin{equation*}
P(E;\mathbb{R}^N\setminus \C) \leq P(F;\mathbb{R}^N\setminus \C)+\Lambda_0v^{-\frac{1}{N}}\big||F|-|E|\big|\quad \forall F\subset \R^N\setminus \C\,.
\end{equation*}
In turn, this implies by Lemma~\ref{lm:density} together with a rescaling argument (see also the proof of Proposition~\ref{prop:ICR}) that 
\beq\label{upper lower density est}
\min\left\{|E\cap B_r(x)|, |B_r(x)\setminus E| \right\}\geq c_0r^N \quad \text{for all $0<r\leq r_0v^{\frac1N}$.}
\eeq
for a suitable $r_0=r_0(N)>0$. 
Suppose now that for some $x\in \mathrm{cl}(\partial E \setminus \C)$,
\begin{equation}\label{achieve case one}
h:=\mathrm{hd}\,(\partial E\setminus \C, \partial B^{(v)}(x_0))=\mathrm{dist}\,(x,\partial B^{(v)}(x_0))>0\,;    
\end{equation}
we will handle the other case for computing $h$ in \eqref{case two h}. Then $B_h(x) \cap \partial B^{(v)}(x_0)=\emptyset$, so due to \eqref{upper lower density est} and the quantitative isoperimetric inequality, we get
\begin{align}\label{case one hausdorff}
    C(N)\sqrt{\delta_{\mathrm{iso}}(E)}\geq \frac{|E\Delta B^{(v)}(x_0)|}{v}\geq 
    c_0\frac{\min\{h,r_0v^{1/N}\}^N }{v}\,.
\end{align}
On the other hand, since $d^*(\C)=k$,  \eqref{isop deficit} reads
\begin{equation}\label{dstar one isop def}
\delta_{\mathrm{iso}}(E)\lesssim \frac{(\mathrm{diam}\, \C \cap \partial E)^k}{v^{(N-1)/N}}\lesssim  v^{-(N-1-k)/N}\,,
\end{equation}
in which case the minimum in \eqref{case one hausdorff} must achieved by $h$ if $v\geq v_1$ and we choose $v_1\geq v_0$ large enough. Combining \eqref{case one hausdorff} and \eqref{dstar one isop def}, we arrive at
\begin{equation}\label{case one multiplied hausdorff}
    h \leq C(N) \delta_{\mathrm{iso}}^{1/(2N)}(E)v^{1/N}\lesssim v^{1/N}\frac{(\mathrm{diam}\, \C \cap \partial E)^{k/(2N)}}{v^{(N-1)/(2N^2)}}\,.
\end{equation}
Since $\mathrm{diam}\, E\leq d_0 v^{1/N}$ by \eqref{diameter bound}, dividing \eqref{case one multiplied hausdorff} by $v^{1/N}$ yields \eqref{hausdorff estimate} in the case that \eqref{achieve case one} holds. Conversely, if there exists $x\in \partial B^{(v)}(x_0)$ such that
\begin{align}\label{case two h}
    h:= \mathrm{hd}\,(\partial E \setminus \C, \partial B^{(v)}(x_0))=\mathrm{dist}\,(x,\partial E \setminus \C)>0\,,
\end{align}
then either $B_h(x) \subset E$ or $B_h(x)\subset \mathbb{R}^N\setminus E$. An analogous argument as in the previous case leads again to \eqref{case one multiplied hausdorff} and thus to \eqref{hausdorff estimate}.\end{proof}\par\medskip
    \begin{remark}\label{enveloping remark}
     Improved estimates on the Hausdorff distance of $\partial E$ to the large ball could be used to determine an upper bound for $\mathcal{R}_\C(v)$. In a nutshell, one would need to prove Lemma~\ref{cylinder hausdorff lemma take 2} with $\partial E$ in \eqref{hausdorff estimate}  instead of $\partial E \setminus \C$. Let us explain why these estimates should hold and then how the argument would go. Assume for simplicity $\C=Z+D$ with $\dim Z=k=d^*(\C)\leq N-2$, $\dim D=N-d^*(\C)$ and $D\subset Z^\perp$. The estimate \eqref{hausdorff estimate} with $\partial E$ would hold if we knew for example that minimizers $E$ never ``envelop" $\C$; that is, there are no slices $\C_z=\C\cap (z+D)$, with $z\in Z$, such that $\partial \C_z\subset \partial E$. If this were true (and it seems like it should be - why should a minimizer envelop the obstacle?), then every point in $\partial E$ would be close to a point in $\partial E \setminus \C$ and thus \eqref{hausdorff estimate} would hold with $\partial E$. Thus we would have
    \begin{align}\notag
    \frac{\mathrm{hd}\,(\partial E , \partial B^{(v)}(x_0))}{v^{1/N}}\lesssim \frac{(\mathrm{diam}\,\C \cap \partial E)^{d^*(\C)/(2N)}}{v^{(N-1)/(2N^2)}}\,.
\end{align}
    Setting $r_v=(v/\omega_N)^{1/N}$ and $h=\mathrm{hd}\,(\partial E, \partial B^{(v)}(x_0))$, 
\begin{equation}\label{containment of E in annulus}
    \partial E \cap \C \subset \partial E \subset \overline{B_{r_v+h}(x_0)}\setminus B_{r_v-h}(x_0)\,.
\end{equation}
From the Pythagorean theorem we know that the longest line segment contained in $\overline{B_{r_v+h}(x_0)}\setminus B_{r_v-h}(x_0)$ has length bounded by $4\sqrt{r_vh}$. Therefore, \eqref{containment of E in annulus} and \eqref{hausdorff estimate} imply that for $v\geq v_1$
\begin{align}\notag
    \mathrm{diam}\,(\C \cap \partial E) \lesssim \sqrt{rh} \lesssim  \sqrt{\frac{(\mathrm{diam}\,\C \cap \partial E)^{d^*(\C)/(2N)}}{v^{(N-1)/(2N^2)}}\,v^{2/N}}\,.
\end{align}
Rearranging this, we find
\begin{align}\notag
    \mathrm{diam}\,(\C\cap \partial E) \lesssim v^{\frac{3N+1)}{N(4N-d^*(\C))}}\,.
\end{align}
Recalling the upper bound in \eqref{first lower  residue bound}, the above estimate would lead to the improved upper bound
$$
\mathcal R_\C(v) \lesssim v^{\frac{(3N+1)d^*(\C)}{N(4N-d^*(\C))}}\,.
$$
Note that $\frac{(3N+1)d^*(\C)}{N(4N-d^*(\C))}<\frac{d^*(\C)}{N}$. Lastly, we remark that this argument is predicated on detailed geometric information on minimizers. Based on \cite{MN}, it stands to reason that any such a resolution of minimizers and their energies would require a $\e$-regularity criterion tailored to the cylindrical geometry of $C=Z + D$. In \cite{MN}, the idea is that, for large volumes, the compact set outside which one is solving an isoperimetric problem is so small relative to the length scale $v^{1/N}$ set by the minimizer that it functions like an isolated singularity of the bounded mean curvature hypersurface $\pa E$. Here this role would instead be played by $C$, which, when the volume is large, acts as the $k$-dimensional subspace $Z$.
\end{remark}

 






\end{document}